\documentclass[11pt,reqno,a4paper]{amsart}
\usepackage{graphicx,color}

\usepackage{amsmath,amsfonts,amssymb,amsthm,mathrsfs,amsopn}
\usepackage{latexsym}

\usepackage[numbers]{natbib}

\usepackage{hyperref}

\usepackage[usenames,dvipsnames,x11names,table]{xcolor}

\numberwithin{equation}{section}

\def\A{\mathcal A}

\def\H{\mathcal H}

\def\R{\mathbb R}
\def\N{\mathbb N}

\newcommand{\dist}{\mathop{\mathrm{dist}}}
\def\e{\varepsilon}
\def\s{\sigma}
\def\S{\Sigma}
\def\vphi{\varphi}

\def\om{\omega}

\def\g{\gamma}

\def\de{\delta}
\def\Id{{\rm Id}}

\def\spt{{\rm spt}}

\def\pa{\partial}

\def\00{{\bf 0}}

\def\F{\mathcal{F}}

\def\P{\mathcal{P}}

\def\CC{\mathcal{C}}

\renewcommand{\a}{\alpha}
\renewcommand{\b}{\beta}

\newcommand{\D}{\Delta}

\renewcommand{\L}{\Lambda}

\renewcommand{\om}{\omega}

\newcommand{\cc}{\subset\subset}

\def\Lip{{\rm Lip}\,}

\newcommand\res{\mathop{\hbox{\vrule height 7pt width .3pt depth 0pt \vrule height .3pt width 5pt depth 0pt}}\nolimits}

\def\weak{\stackrel{*}{\rightharpoonup}}

\def\C{\mathbf{C}}
\def\D{\mathsf{D}}

\def\Id{\mathrm{Id}}

\theoremstyle{plain}
\newtheorem{theorem}{Theorem} [section]

\newtheorem{lemma}[theorem]{Lemma}

\theoremstyle{definition}
\newtheorem{definition}[theorem]{Definition}
\newtheorem{remark}[theorem]{Remark}
\newtheorem*{ack}{Acknowledgements}

\title{A direct approach to Plateau's problem in any codimension}

\author{G. De Philippis}
\address{Institut f\"ur Mathematik, Universitaet Z\"urich, Winterthurerstrasse 190, CH-8057 Z\"urich, Switzerland}
\email{guido.dephilippis@math.uzh.ch}

\author{A. De Rosa}
\address{Institut f\"ur Mathematik, Universitaet Z\"urich, Winterthurerstrasse 190, CH-8057 Z\"urich, Switzerland}
\email{antonio.derosa@math.uzh.ch}

\author{F. Ghiraldin}
\address{Institut f\"ur Mathematik, Universitaet Z\"urich, Winterthurerstrasse 190, CH-8057 Z\"urich, Switzerland}
\email{francesco.ghiraldin@math.uzh.ch}

\begin{document}

\begin{abstract} 

 This paper aims to propose a direct approach to solve the  Plateau's problem in codimension higher than one. The problem is 
 formulated as the minimization of the Hausdorff measure among a family of \(d\)-rectifiable closed subsets of $\R^n$: 
 following the previous work \cite{DelGhiMag} the existence result is obtained by a compactness principle valid under fairly general assumptions on the class of competitors. Such class is then specified to give meaning to boundary conditions. 
We also show  that the obtained minimizers are regular up to a set of dimension less than \((d-1)\). 
\end{abstract}

\maketitle

\section{Introduction}\label{intro}

Plateau's problem consists in looking for a surface of minimal area among those surfaces spanning a given boundary. A considerable amount of the development of Geometric Measure Theory in the last fifty years has been devoted to provide generalized concepts of surface, area and  of ``spanning a given boundary'', in order to apply the direct methods of the calculus of variations to the Plateau's problem. In particular we recall the notions of sets of finite perimeter \cite{DeGiorgiSOFP1,DeGiorgiSOFP2}, of currents \cite{federer60} and of varifolds \cite{Allard, Allardboundary, Almgren68}, introduced respectively by De Giorgi, Federer, Fleming, Almgren and  Allard. A more ``geometric'' approach was proposed by Reifenberg in \cite{reifenberg1}, where Plateau's problem was set as the minimization of Hausdorff \(d\)-dimensional measure among compact sets and the notion of spanning a given boundary was given in term of inclusions of homology groups.

Any of these approach has some drawbacks: in particular, not all the ``reasonable'' boundaries can be obtained by the above notions and not always the solutions are allowed to have the type of singularities observed by soap bubble (the so called Plateau's laws). Recently  in \cite{HarrisonPugh14} Harrison and Pugh, see also \cite{Harrison2014}, proposed a new notion of spanning a boundary, which seems to include all reasonable physical boundaries and they have been able to show, in the co-dimension one case, existence of least area surfaces spanning a given boundary.

In the recent paper \cite{DelGhiMag},  De Lellis, Maggi and the third author have proposed a direct approach to the Plateau's problem, based on the ``elementary'' theory of Radon measures and on a deep result of Preiss concerning rectifiable measures.  Roughly speaking they showed, in the co-dimension one case, that every time one has a class which contains ``enough'' competitors (namely the cone and the cup  competitors, see \cite[Definition 1]{DelGhiMag}) it is always possible to show that the infimum of the Plateau's problem is achieved by the area of a  rectifiable set. They then applied this result to provide a new proof of Harrison and Pugh theorem as well as to show  the existence of sliding minimizers, a new notion of minimal sets proposed by David in \cite{davidplateau,DavidBeginners} and inspired by Almgren's \((\mathbf{M},0,\infty)\),  \cite{Almgren76}.

\medskip

In this note, we extend the result \cite{DelGhiMag} to any co-dimension. More precisely, we prove that every time the class of competitors for the Plateau's Problem consists of rectifiable sets and it is closed by Lipschitz deformations, it is always possible to show that the infimum is achieved by a  compact set \(K\) which is, away from the ``boundary'',  an analytic manifold outside a closed set of Hausdorff dimension at most \((d-1)\), see Theorem \ref{thm generale} below for the precise statement. We then apply this result to provide existence of sets spanning a given boundary according to the natural generalization of the notion introduced by Harrison and Pugh, Theorem \ref{thm generale}, and to show the existence of sliding minimizers in any co-dimension, Theorem \ref{thm david}.

Although the general strategy of the proof is the same of \cite{DelGhiMag}, some non-trivial modifications have to be done in order to deal with sets of any co-dimension. In particular, with respect to \cite{DelGhiMag}, we use a different notion of ``good class'', the main reason being the following:  one of the key step of the proof of  our main result consists in showing a precise density lower bound for the measure obtained as limit of the sequence of Radon measures naturally associated to a minimizing sequence \((K_j)\), see Steps 1  and 4 in the proof of Theorem \ref{thm generale}. In order to obtain such a lower bound, instead of relying on relative isopermetric inequalities on the sphere as in \cite{DelGhiMag} (which are peculiar of the co-dimension one case)  we use the deformation theorem of David and Semmes in \cite{DavidSemmes} to obtain suitable competitors,  following a strategy already introduced by  Federer and Fleming for rectifiable currents, see \cite{federer60} and \cite{Almgren76}.  Moreover since our class is essentially closed by Lipschitz deformations,  we are actually  able to  prove that any set achieving the infimum is  a stationary varifold and that, in addition, it is  smooth outside a closed set of relative co-dimension one (this does not directly follows by Allard's regularity theorem, see Step 7 in the proof of Theorem \ref{thm generale}). Simple examples show that this regularity is actually optimal.

\medskip

In order to precise  state our main results, let us introduce some notations and definitions, referring to Section \ref{notation} for more details. We will always work in \(\R^n\) and \(1\le d\le n \) will always be an integer number, we recall that a set \(K\) is said to be \(d\)-rectifiable if it can be covered up to an \(\H^d\) negligible set by countably many \(C^1\) manifolds, see \cite[Chapter 3]{SimonLN}, where \(\H^d\) is the \(d\)-dimensional Hausdorff measure.  We also  let $\mathsf{Lip}(\R^n)$ be the space of Lipschitz maps in $\R^n$.

\begin{definition}[Lipschitz deformations]\label{d:deform}
Given a ball $B_{x,r}$, we let $\mathfrak D(x,r)$ be the set of functions $\vphi:\R^n \rightarrow \R^n$ 
such that $\varphi(z)=z$ in $\R^n\setminus B_{x,r}$ and which are smoothly isotopic 
to the identity inside $B_{x,r}$, namely those for which there exists a smooth isotopy $\lambda:[0,1]\times \R^n\rightarrow\R^n$ such that 
$$\lambda(0,\cdot) = \Id, \quad \lambda(1,\cdot)=\vphi, \quad \lambda(t,h)=h \quad\forall\,(t,h)\in [0,1]\times (\R^n \setminus B_{x,r}) \quad \mbox{ and } $$
$$ \lambda(t,\cdot) \mbox{\ is a diffeomorphism of } \R^n \ \forall t \in [0,1].
$$
We finally set $\D(x,r):=\overline{\mathfrak D(x,r)}^{C^0}\cap \mathsf{Lip}(\R^n)$, the intersection of the Lipschitz maps with the closure of $\mathfrak D(x,r)$ with respect to the uniform topology.
\end{definition}

The following definition describes the properties we require to the comparison sets: the key property we ask 
for $K'$ to be a competitor of $K$ is that $K'$ must be close in energy to sets obtained from $K$ via deformation maps in Definition \ref{d:deform}. 
This allows a larger flexibility on the choice of the admissible sets, since a priori $K'$ might not belong to the competition class.

\begin{definition}[Deformed competitors and good class]\label{def good class}
Let $H\subset \mathbb R^{n}$ be closed. 
Given $K\subset \mathbb R^{n}\setminus H$ relatively closed countably $\H^d$-rectifiable and  $B_{x,r}\subset \mathbb R^{n}\setminus H$, 
a {\em deformed competitor} for $K$ in $B_{x,r}$ is any set of the form
\begin{equation*}
 \varphi \left (  K \right ) \quad \mbox{ where } \quad \varphi \in \D(x,r).
\end{equation*}

Given a family $\mathcal{P} (H)$ of relatively closed \(d\)-rectifiable subsets $K\subset \mathbb R^{n}\setminus H$, 
we say that $\mathcal{P} (H)$ is a {\em good class} if for every $K\in\mathcal{P} (H)$, for every $x\in K$ 
and for a.e. $r\in (0, \dist (x, H))$
\begin{equation}
  \label{inf good class}
  \inf \big\{ \H^d (J) : J\in \mathcal{P} (H)\,,J\setminus \overline{B_{x,r}} =K\setminus \overline
{B_{x,r}} \big\} \leq \H^d (L)\,
\end{equation}
whenever $L$ is any deformed competitor for $K$ in $B_{x,r}$.
\end{definition}

Once we fix a closed set $H$, we can formulate Plateau's problem in the class $\mathcal{P}(H)$:
\begin{equation}  \label{plateau problem generale}
m_0 :=  \inf \big\{\H^d (K) : K\in \mathcal{P}(H)\big\}\,.
\end{equation}
We will say that a sequence $(K_j) \subset \mathcal{P} (H)$ is a {\em minimizing sequence} if   $\mathcal{H}^d (K_j) \downarrow m_0$. The following  theorem is our main result and establishes the behavior of minimizing sequences.

\begin{theorem}\label{thm generale}
Let $H\subset \mathbb R^{n}$ be closed and $\mathcal{P} (H)$ be a good class. Assume the infimum in Plateau's problem \eqref{plateau problem generale} is finite and let 
$(K_j)\subset \mathcal{P}(H)$ be a minimizing sequence. Then, up to subsequences, the measures $\mu_j := \H^d \res K_j$ converge weakly$^\star$ in $\mathbb R^{n}\setminus H$ 
to a measure $\mu = \H^d \res K$, where $K = \spt\, \mu \setminus H$ is a countably $\H^d$-rectifiable set. Furthermore: 
\begin{itemize}
\item[(a)] the integral varifold naturally associated to $\mu$ is stationary in $\R^{n}\setminus H$;
\item[(b)] $K$ is a  real analytic submanifold outside a relatively closed set $\S\subset K$ with 
$\dim_\H(\S) \leq d-1$.
\end{itemize}
In particular, $\liminf_j\H^d(K_j)\ge \H^d(K)$ and if $K \in \mathcal{P} (H)$, then $K$ is a minimum for \eqref{plateau problem generale}.
\end{theorem}

We wish to apply Theorem \ref{thm generale} to two definitions of boundary conditions. The first one is the natural generalization of the one 
considered in \cite{HarrisonPugh14}:
\begin{definition}\label{def plateau first}
Let $H$ be a closed set in $\R^{n}$. 

Let us consider the family
\[
\CC_H=\big\{\g:S^{n-d}\to\R^{n}\setminus H:\mbox{$\g$ is a smooth embedding of $S^{n-d}$ into $\R^{n}$}\big\}\,.
\]
We say that $\CC\subset\CC_H$ is closed by isotopy (with respect to $H$) if $\CC$ contains all elements
$\gamma'\in \CC_H$ belonging to the same smooth isotopy class $[\g]$ in $\R^{n}\setminus H$ 
of any $\gamma \in \CC$, see \cite[Ch. 8]{Hirsch94}. Given $\CC\subset\CC_H$ closed by isotopy, we say that a relatively closed subset $K$ of $\R^{n}\setminus H$ is a $\CC$-spanning set of $H$ if
\begin{equation*}
\mbox{$K\cap\g\ne\emptyset$ for every $\g\in\CC$}\,.
\end{equation*}
We denote by $\mathcal{F} (H, \CC)$ the  family of countably $\mathcal{H}^d$-rectifiable sets which are $\CC$-spanning sets of $H$.
\end{definition}

We can prove the following closure property for the class $\F(H,\CC)$:
\begin{theorem}\label{thm plateau}
Let $H$ be closed in $\R^{n}$ and $\CC$ be closed by isotopy with respect to $H$,  then:
\begin{itemize}
\item[(a)]  $\F (H,\CC)$ is a good class in the sense of Definition \ref{def good class}. 
\item[(b)] Assume the infimum \eqref{plateau problem generale} corresponding to $\P(H)=\F(H,\CC)$ is finite, then the set \(K\) provided by Theorem \ref{thm generale} belongs to \(\F(H,\CC)\). In particular  the Plateau's problem in the class \(\F(H,\CC)\) has a solution.
\end{itemize}

\end{theorem}

The second type of boundary condition we want to consider is the one related to the notion of  ``sliding minimizers''  introduced by David in \cite{davidplateau,DavidBeginners}.
\begin{definition}[Sliding minimizers]
  Let $H\subset \R^{n}$ be closed and $K_0 \subset \R^{n}\setminus H$ be relatively closed.  We denote by $\S(H)$ the family of Lipschitz maps $\vphi:\R^{n}\to\R^{n}$ such that there exists a continuous map $\Phi:[0,1]\times\R^{n}\to\R^{n}$ with $\Phi(1,\cdot)=\vphi$, $\Phi(0,\cdot)=\Id$ and $\Phi(t,H)\subset H$ for every $t\in[0,1]$. We then define
  \[
  \A(H,K_0)=\big\{K:\mbox{$K=\vphi(K_0)$ for some $\vphi\in\S(H)$}\big\}\,
  \]
and say that $K_0$ is a sliding minimizer if $\H^d(K_0)=\inf\{\H^d(J):J\in\A(H,K_0)\}$.
\end{definition}

\begin{remark}\label{remark david}
For every $K_0 \subset \R^{n}\setminus H$  
relatively closed and  $d$-rectifiable, $\A(H,K_0)$ is a good class in the sense of Definition \ref{def good class}, 
since $\D(x,r)\subset\Sigma(H)$ for every $B_{x,r}\subset\R^n\setminus H$.
\end{remark}

Applying Theorem \ref{thm generale} to the contest of sliding minimizers we obtain the following result which is the analogous of \cite[Theorem 7]{DelGhiMag} in any codimension. Here and in the following $U_\delta (E)$ denotes the  $\delta$-neighborhood of a set  $E\subset \R^n$.
\begin{theorem}\label{thm david}
 Assume that
\begin{itemize}
\item[(i)] $K_0$ is bounded  $d$-rectifiable with $\H^d(K_0)<\infty$;
\item[(ii)] $\H^d(H)=0$ and for every $\eta>0$ there exist $\de>0$ and $\Pi\in\S(H)$ such that
  \begin{equation*}
    \label{retraction}
      \Lip\Pi\le1+\eta\,,\qquad \Pi(U_\de(H))\subset H\,.
  \end{equation*}
\end{itemize}
Then, given any minimizing sequence $(K_j)$ in the Plateau's problem corresponding to $\P(H)=\A(H,K_0)$ and any set $K$ as in Theorem \ref{thm generale}, we have
  \begin{equation*}
     \inf \big\{\H^d (J) : J\in \A (H, K_0)\big\} =\H^d(K)=\inf\big\{\H^d(J):J\in\A(H,K)\big\}\,.
  \end{equation*}
In particular $K$ is a sliding minimizer.
\end{theorem}

The paper is structured as follows, in Section \ref{notation} we will recall some basic definitions and recall some known theorems we are going to use, in particular Preiss rectifiability criterion and a version of the deformation theorem due to David and Semmes. In Section \ref{mainresult} we prove Theorem \ref{thm generale} and in Section \ref{harrisonanddavid} we prove Theorems \ref{thm plateau} and \ref{thm david}.

\begin{ack}
The authors are grateful to Camillo De Lellis, Francesco Maggi and Emanuele Spadaro  for many interesting comments and suggestions. This work has been supported by ERC 306247 {\it Regularity of area-minimizing currents} and by
SNF 146349 {\it Calculus of variations and fluid dynamics}. 
\end{ack}

\section{Notation and preliminaries}\label{notation}
We are going to use the following notations:  $Q_{x,l}$  denotes the closed cube centered in $x$, 
with edge length $l$; moreover we set 
\[
R_{x,a,b}:=x+\Big[-\frac a2,\frac a2\Big]^d \times \Big[-\frac b2,\frac b2\Big]^{n-d}\quad \textrm{ and} \quad B_{x,r}:=\{y \in \R^n \ : \ |y-x|<r\}.
\]
When cubes, rectangles and balls are centered in the origin, we will simply write  $ Q_{l}, \ R_{a,b}$ and $B_{r}$. Cubes and balls in the subspace $\R^d\times\{0\}^{n-d}$ are denoted with $Q_{x,l}^d$ and $B_{x,r}^{d}$ respectively. We also 
let $\omega_d$ be the Lebesgue measure of the unit ball in $\mathbb R^d$.

Let us recall the following deep structure result for Radon measures due to Preiss 
\cite{Preiss, DeLellisNOTES} which will play a key role in the proof of Theorem \ref{thm generale}.
\begin{theorem}\label{Preiss}
Let $d$ be an integer and $\mu$ a locally finite measure on $\R^n$ such that the $d$-density
$$
\theta(x):= \lim_{r\rightarrow 0}\frac{\mu(B_{x,r})}{\omega_d r^d}
$$
exists and satisfies $0<\theta(x)<+\infty$ for $\mu$-a.e. $x$. Then $\mu = \theta \H^d\res K$, where $K$ is a countably $\H^d$-rectifiable set. 
\end{theorem}

In order to apply Preiss' Theorem  we will rely on the monotonicity formula for minimal surfaces, which roughly speaking can be obtained by comparing the given minimizer with a cone. To this aim let us introduce the following definition:

\begin{definition}[Cone competitors]
In the setting of Definition \ref{def good class}, the cone competitor for $K$ in $B_{x,r}$ is the following set
\begin{equation}
  \label{cone comp}
\C_{x,r}(K) = \big(K\setminus B_{x,r}\big) \cup \big\{\lambda x+(1-\lambda)z:z\in K\cap\pa B_{x,r}\,,\lambda\in[0,1]\big\}\,.
\end{equation}
\end{definition}
Let us note that in general  a cone competitor in $B_{x,r}$ is not a deformed competitor in $B_{x,r}$. 
On the other hand as in \cite{DelGhiMag} we can show that:

\begin{lemma}\label{l:10}
Given a good class $\mathcal{P} (H)$ in the sense of Definition  \ref{def good class}, for any $K\in \mathcal{P} (H)$  countably $ \mathcal{H}^d\mbox{-rectifiable}$ and for every $x\in K$, the set $K$ verifies for a.e. $r\in (0, \dist (x, H))$:
\begin{equation*}
  \inf \big\{ \H^d (J) : J\in \mathcal{P} (H)\,,J\setminus \overline {B_{x,r}} =K\setminus \overline
{B_{x,r}} \big\} \leq \H^d (\C_{x,r}(K)). 
\end{equation*}
\end{lemma} 
 \begin{proof}
Without loss of generality let us consider balls $B_r$ centered at $0$ with $B_r\cc\R^{n}\setminus H$. We assume in addition that $K\cap\pa B_r$ is $\H^{d-1}$-rectifiable with $\H^{d-1}(K\cap\pa B_r)<\infty$ and that $r$ is a Lebesgue point of $t\in(0,\infty)\mapsto \H^{d-1}(K\cap\pa B_t)$. All these conditions are fulfilled for a.e. $r$ and, again by scaling, we can assume that $r=1$ and use $B$ instead of $B_1$. For $s\in(0,1)$ let us set
\[
\vphi_s(r)=\left\{
\begin{array}{l l}
  0\,,&r\in[0,1-s)\,,
  \\
  \frac{r-(1-s)}s\,,&r\in[1-s,1]\,,
  \\
  r\,,&r\ge 1\,,
\end{array}
\right .
\]
and $\phi_s(x)=\vphi_s(|x|)\frac{x}{|x|}$ for $x\in\R^{n}$. In this way, one easily checks that  $\phi_s:\R^{n}\to\R^{n} \in \D(0,1)$. 

Since $\phi_s(K\cap \overline{B_{1-s}})=\{0\}$, we need to show that
\[
\limsup_{s\to 0^+}\H^d\big(\phi_s(K\cap (B\setminus B_{1-s}))\big)\le\frac{\H^{d-1}(K\cap\pa B)}d = \H^d(\C_{x,r}(K))\,.
\]
Let $x_0\in K\cap\partial B_t$ and let us fix an orthonormal base $\nu_1,\dots,\nu_d$ of 
the approximate tangent space $T_{x_0}K$ such that 
$\nu_i\in T_{x_0}K\cap T_{x_0}\partial B_t$ for $i\leq d-1$. 
Let 
$$
J_d^K\phi_s = \Big|(\bigwedge^d D\phi_s) (T_{x_0}K)\Big| = |D\phi_s(\nu_1)\wedge\dots\wedge D\phi_s(\nu_d)|
$$ 
be the $d$-dimensional tangential Jacobian of $\phi_s$ with respect to $K$.  Letting $I$ be the (at most countable) set of those $t\in(0,1)$ such that $\H^{d-1}(K\cap\pa B_t)>0$, we find with the aid of the area and co-area formulas,
\begin{eqnarray}\nonumber
  &&\H^d\big(\phi_s(K\cap (B\setminus B_{1-s}))\big)=\int_{K\cap(B\setminus B_{1-s})}J_d^K\phi_s\,d\H^d
  \\\label{ignoriamo}
  &=&\int_{1-s}^1\,dt\int_{K\cap\pa B_t}\frac{J_d^K\phi_s}{|\nu_d\cdot\hat x|}\,d\H^{d-1}+ \sum_{t\in I\cap(1-s,1)}\left(\textstyle{\frac{t- (1-s)}{1-s}}\right)^d\H^d( K\cap\pa B_t),
\end{eqnarray}
where $\hat x=x/|x|$ and we have used that \(|\nu_d\cdot\hat x|\) is the tangential  co-area factor of the map \(f(x)=|x|\).
We first notice that, for $t\in (1-s,1)$, $t-(1-s) \leq s$. Moreover
\[
\lim_{s\to 0}\sum_{t\in I\cap(1-s,1)}\H^d(K\cap\pa B_t)=0\,,
\]
and thus the second term in \eqref{ignoriamo} can be ignored. At the same time, for a constant $C$,
\[
J_d^K\phi_s(x)\le C+|\hat x\cdot\nu_d|\,\vphi_s'(|x|)\,\Big(\frac{\vphi_s(|x|)}{|x|}\Big)^{d-1}\,,\qquad\mbox{for $\H^d$-a.e. $x\in K$}\,.
\]
The constant $C$ gives a negligible contribution in the integral as $s\downarrow 0$; as for the second term, having $\vphi_s'=1/s$ on $(1-s,1)$, we find
\[
\int_{1-s}^1 \H^{d-1}(K\cap\pa B_t)\,\vphi_s'(t)\,\Big(\frac{\vphi_s(t)}{t}\Big)^{d-1}\,dt
=\frac1s\int_{1-s}^1 \H^{d-1}(K\cap\pa B_t)\,\Big(\frac{\vphi_s(t)}{t}\Big)^{d-1}\,dt\,.
\]
Since $t=1$ is a Lebesgue point of $t\in(0,\infty)\mapsto \H^{d-1}(K\cap\pa B_t)$, we have
\[
\lim_{s\to 0}\frac1s\int_{1-s}^1 |\H^{d-1}(K\cap\pa B_t)-\H^{d-1}(K\cap\pa B)|\,dt=0\,,
\]
so that, combining the above remarks we find
\[
\limsup_{s\to 0^+}\H^d(\phi_s(K\cap B))\le \H^{d-1}(K\cap\pa B)\,
\limsup_{s\to 0^+}\frac1s\int_{1-s}^1 \Big(\frac{\vphi_s(t)}{t}\Big)^{d-1}\,dt=\frac{\H^{d-1}(K\cap\pa B)}d\,,
\]
as required. 
 \end{proof}

The second key result we are going to use is  a deformation theorem for closed sets due to David and Semmes  \cite[]{DavidSemmes}, 
analogous to the one for rectifiable currents \cite{SimonLN,FedererBOOK}. We provide a slightly extended statement for the sake of forthcoming proofs. 

Before stating the theorem, let us introduce some further notation.
Given a closed cube $Q$ of edges length $l$ in $\R^n$ and $\e > 0$, we 
cover $Q$ with  closed smaller cubes 
with edges length $\e \ll l$, non empty intersection with $\mbox {Int} (Q)$ 
and such that the decomposition is centered in $x$ (i.e. one of the 
subcubes is centered in $x$). The family of this smaller cubes is denoted $\L_\e(Q)$. 
We set 
\begin{equation}\label{cornici}
\begin{gathered}
C_1:=\bigcup \left \{T \cap Q: T \in \L_\e(Q), T \cap \pa Q \neq \emptyset  \right \},\\
C_2:=  \bigcup \left \{T \in \L_\e(Q) :  T \not \subset C_1, T \cap \pa C_1  \neq \emptyset \right \},\\
Q_1:=\overline{Q \setminus (C_1 \cup C_2)}
\end{gathered}
\end{equation}
and consequently
$$\L_\e(Q_1 \cup C_2):= \left \{T \in \L_\e(Q) : T \subset (Q_1 \cup C_2)\right \}$$
For each nonnegative integer $m\leq n$, let $\L_{\e,m}(Q_1 \cup C_2)$ denote the collection 
of all $m$-dimensional faces of cubes in $\L_\e(Q_1 \cup C_2)$ and 
$\L^*_{\e,m}(Q_1 \cup C_2)$ will be the set of the elements of $\L_{\e,m}(Q_1 \cup C_2)$ 
which are not contained in $\pa (Q_1 \cup C_2)$. We also let  $S_{\e,m}(Q_1 \cup C_2):= \bigcup \L_{\e,m}(Q_1 \cup C_2)$ be the 
 the $m$-skeleton of order $\e$ in $Q_1 \cup C_2$.

\begin{theorem}\label{deformationcube}
Let $r>0$ and $E$ be a compact subset of $Q$  
such that $\H^d(E)<+\infty$ and $Q \subset B_{x,r}$. 
There exists a map $\Phi_{\e,E} \in \D(x,r)$ 
 satisfying the following properties:
\begin{itemize}
\item[(1)]  $\Phi_{\e,E}(x)=x \ \mbox{for} \ x \in \R^n \setminus (Q_1 \cup C_2)$;
\item[(2)]  $\Phi_{\e,E}(x)=x \ \mbox{for} \ x \in S_{\e,d}(Q_1 \cup C_2)$;
\item[(3)]  $\Phi_{\e,E}(E)\subset S_{\e,d}(Q_1 \cup C_2)\cup \pa(Q_1 \cup C_2)$;
\item[(4)]  $\Phi_{\e,E}(T)\subset T$ for every $T \in \L_{\e,m}(Q_1 \cup C_2)$, with $m=d,...,n$;
\item[(5)] either $\H^d(\Phi_{\e,E}(E\cap T))=0$ or $\H^d(\Phi_{\e,E}(E\cap T))=\H^d(T)$, for every $T \in \L^*_{\e,d}(Q_1)$;
\item[(6)]  $\H^d(\Phi_{\e,E}(E\cap T)) \leq k_1\H^d(E\cap T)$ for every $T \in \L_{\e}(Q_1 \cup C_2)$;
\end{itemize}
where $k_1$ depends only on $n$ and $d$ (but not on $\e$).
\end{theorem}
\begin{proof}
Proposition $3.1$ in \cite{DavidSemmes} provides a map $\widetilde{\Phi_{\e,E}}\in \D(x,r)$ 
 satisfying properties (1)-(4) and (6). 
We want to set $$\Phi_{\e,E}:=\Psi \circ \widetilde{\Phi_{\e,E}},$$
where $\Psi$ will be defined below.
We first define $\Psi$ on every $T \in \L_{\e,d}(Q_1 \cup C_2)$ distinguishing two cases
\begin{itemize}
\item[(a)]  if either $\H^d(\widetilde{\Phi_{\e,E}}(E)\cap T)=0$ or  $\H^d(\widetilde{\Phi_{\e,E}}(E)\cap T)=\H^{d}(T)$ or $T \not \in \L^*_{\e,d}(Q_1)$, then we set $\Psi_{|T}=\Id$;
\item[(b)]  otherwise since $\widetilde{\Phi_{\e,E}}(E)$ is compact, there exists $y_T \in T$ and $\delta_T>0$ such that $B_{\delta_T}(y_T)\cap \widetilde{\Phi_{\e,E}}(E)=\emptyset$; we define 
$$\Psi_{|T}(x)=x+\alpha(x-y_T) \min \left \{1,\frac{|x-y_T|}{\delta_T}\right\},$$
where $\alpha>0$ such that the point $x+\alpha(x-y_T)\in  \left (\pa T\right ) \times \{0\}^{n-d}$.
\end{itemize}
The second step is to define $\Psi$ on every $T' \in \L_{\e,d+1}(Q_1 \cup C_2)$. Without loss of generality we can assume $T'$ centered in $0$. 
We divide $T'$ in pyramids $P_{T,T'}$ with base $T\in \L_{\e,d}(Q_1 \cup C_2)$ and vertex $0$. Assuming $T \subset \{x_{d+1}=-\frac \e 2, x_{d+2},...,x_n=0\}$ and $T' \subset \{x_{d+2},...,x_n=0\}$, we set
$$\Psi_{|P_{T,T'}}(x)=-\frac{2x_{d+1}}{\e}\Psi_{|T}\left(-\frac{x}{x_{d+1}}\frac \e 2\right).$$
We iterate  this procedure on all the dimensions till to $n$, defining it well in $Q_1 \cup C_2$. 
Since $\Psi_{|\pa (Q_1 \cup C_2)}=\Id$ we can extend the map as the identity outside $Q_1 \cup C_2$.
In addition   one can easily check that $\Psi \in \D(x,r)$ and thus, since  \(\widetilde{\Phi_{\e,E}}\in \D(x,r)\) and the class \( \D(x,r)\)  is closed by composition, this concludes the proof.
\end{proof}

Later we will need to implement the above deformation of a set $E$ on a rectangle rather than a cube: 
the deformation theorem can be proved for very general cubical complexes, \cite{almgren86}; however, for the sake of exposition, we limit ourselves to note the following simple observation: given a closed rectangle $R:=R_{x,a,b}$, using a linear map, 
and covering this time with rectangles homothetic to $R$, one can easily deduce the same thesis as in Theorem \ref{deformationcube}. 
The only key point is the area estimate (6), which holds with a constant $k_1$ depending on the ratio $a/b$. We will apply this construction 
to rectangles where the ratio is approximately $2$.

\section{Proof of Theorem \ref{thm generale}}\label{mainresult}
\begin{proof}
  [Proof of Theorem \ref{thm generale}] Up to extracting subsequences we can assume the existence of a Radon measure $\mu$ on $\R^n\setminus H$ such that
\begin{equation}\label{muj va a mu}
  \mu_j\weak\mu\,,\qquad\mbox{as Radon measures on $\R^n\setminus H$}\,,
\end{equation}
where $\mu_j=\H^d \res K_j$. We set $K = \spt\, \mu \setminus H$ and divide the argument in several steps.

\medskip

\noindent {\it Step one}: We show the existence of $\theta_0=\theta_0(n,d)>0$ such that
\begin{equation}
  \label{lower density estimate mu ball}
 \mu(B_{x,r})\ge\theta_0\,\omega_d r^d\,,\qquad \textrm{ \(x\in\spt\,\mu\) and \( r<d_x :=\dist(x,H)\)}
\end{equation}
To this end it is  sufficient to prove the existence of $\beta=\beta(n,d)>0$ such that
\begin{equation*}
  \mu(Q_{x,l})\ge\beta\, l^d\,,\qquad \textrm{ \(x\in\spt\,\mu\) and \( l<2d_x/\sqrt{n}\)}\, .
\end{equation*}
Let us  assume by contradiction that there exist $x\in\spt\,\mu$ and $l<2d_x/\sqrt{n}$ such that 
\begin{equation*}
 \frac{\mu(Q_{x,l})^\frac{1}{d}}{l}<\b.
 \end{equation*}
We claim that this assumption, for $\beta$ chosen sufficiently small depending only on \(d\) and \(n\), implies that for some $l_\infty \in (0,l)$
\begin{equation}
 \label{absurd lower bound}
   \mu(Q_{x,l_\infty})=0,
\end{equation}
which is a contradiction with the property of $x$ to be a point of $\spt\,\mu$.
In order to prove \eqref{absurd lower bound}, we assume that $\mu(\pa Q_{x,l})=0$, which is true for a.e. $l$. 

To prove \eqref{absurd lower bound} we construct a sequence of nested cubes $Q_i = Q_{x,l_i}$ such that, if $\beta$ is sufficiently small 
 it holds:
\begin{itemize}
 \item[(i)] $Q_0 = Q_{x,l}$;
 \item[(ii)] $\mu(\pa Q_{x,l_i})=0$;
 \item[(iii)] setting  $m_i:=\mu( Q_i)$ it holds:
 $$
 \frac{m_i^{\frac{1}{d}}}{l_i}<\beta;
 $$
 \item[(iv)] $m_{i+1}\leq(1-\frac{1}{k_1})m_i$, where $k_1$ is the constant in Theorem \ref{deformationcube} (6);
 \item[(v)] $(1-4\e_i)l_i\geq l_{i+1}\geq (1-6\e_i)l_i$, where 
 \begin{equation}\label{eq:epsi}
 \e_i:=\frac{1}{k\beta}\frac{m_i^{\frac{1}{d}}}{l_i}
 \end{equation}  
and  \(k=\max\{6,6 / (1-(\frac{k_1-1}{k_1})^\frac{1}{d})\}\) is a universal  constant. 
 \item [(vi)]   $\lim_i m_i =0$ and $\lim_i l_i >0$.
 \end{itemize}
Following \cite{DavidSemmes}, we are going to construct the sequence of cubes by induction: the cube $Q_0$ satisfies by construction hypotheses (i)-(iii). Suppose that cubes until step $i$ are already defined. 

Setting  $m_i^j:=\H^d(K_j \cap Q_i)$ we cover $Q_i$ with the family $\L_{\e_i l_i}(Q_i)$ of closed cubes 
with edge length $\e_i l_i$ as described in Section \ref{notation} and we set \(C_1^i\) and \(C_2^i\) for the corresponding sets defined in \eqref{cornici}. We define $Q_{i+1}$ to be the internal cube given by 
the construction, and we note that  $C^i_2$ and $Q_{i+1}$ are non-empty if, for instance,
\[ 
   \e_i =  \frac{1}{k\beta}\frac{m_i^{\frac{1}{d}}}{l_i} <\frac{1}{k}\le\frac{1}{6} ,
\]
which is guaranteed by our choice of $k$. Observe moreover that $C^i_1\cup C^i_2$ is a strip of 
width at most $2\e_i l_i$ around $\pa Q_i$, hence the side $l_{i+1}$ of $Q_{i+1}$ satisfies $(1-4\e_i)l_i\leq l_{i+1} <(1-2\e_i) l_i $. 

Now we apply Theorem \ref{deformationcube} to $Q_i$ with $E=K_j$ and $\e=\e_i l_i$, obtaining the map $\Phi_{i,j}=
\Phi_{K_j,\e_i l_i}$. We claim that, for every $j$ sufficiently large,
\begin{equation} \label{conditionji}
  m_i^j \leq k_1 (m_i^j-m_{i+1}^j) + o_j(1).
\end{equation}
Indeed, since $(K_j)$ is a minimizing sequence, by the definition of good class we have that
\begin{equation*}
\begin{split}
  m_i^j &\leq  m_i + o_j(1)\leq \H^d \left(\Phi_{i,j} \left(K_j \cap Q_0 \right) \right) +o_j(1) \\
     &=  \H^d \left( \Phi_{i,j} \left(K_j \cap Q_{i+1} \right) \right) + \H^d \left( \Phi_{i,j} \left(K_j \cap (C^i_1\cup C^i_2) \right) \right)+o_j(1)\\
        &\leq  k_1\H^d \left( K_j \cap (C^i_1\cup C^i_2)  \right) +o_j(1) = k_1 (m_i^j-m_{i+1}^j) + o_j(1).
        \end{split}
\end{equation*}
The last inequality holds because $\H^d \left( \Phi_{i,j} \left(K_j \cap Q_{i+1} \right) \right)=0$, 
otherwise by property (5) of Theorem \ref{deformationcube} there would exists $T \in \L^*_{\e_i l_i,d}(Q_{i+1})$ such that $\H^d(\Phi_{i,j}(K_j\cap T))=\H^d(T)$. 
Together with property (ii) this would imply
\begin{equation*}
l_i^d\e_i^d = \H^d(T) \leq \H^d\left( \Phi_{i,j}\left(K_j\right) \cap Q_i \right) \leq k_1\H^d\left( K_j \cap Q_i  \right) \leq k_1 m_i^j \rightarrow k_1 m_i
\end{equation*}
and therefore substituting \eqref{eq:epsi}
$$ \frac{m_i}{ k^d\beta^d} \leq k_1 m_i,$$
which is false if $ \beta$ is sufficiently small ($m_i>0$ because $x\in\spt(\mu)$). 
Passing to the limit in $j$ in \eqref{conditionji} we obtain (iv): 
\begin{equation}\label{eq:stimai}
m_{i+1}\leq \frac{k_1-1}{k_1} m_i. 
\end{equation}
Since $l_{i+1}\geq(1-4\e_i)l_i$ we can slightly shrink the cube $Q_{i+1}$ to a concentric cube $Q'_{i+1}$ 
with $l'_{i+1}\geq(1-6\e_i) l_i>0$, $\mu(\pa Q'_{i+1})=0$ and for which (iv) still holds, since just $m_{i+1}$ decreased. 
With a slight abuse of notation we rename this last cube $Q'_{i+1}$ as $Q_{i+1}$.

We now show (iii). Using \eqref{eq:stimai} and condition (iii) for $Q_i$ we obtain
\begin{equation*}
\frac{m_{i+1}^\frac{1}{d}}{l_{i+1}} \leq \left(\frac{k_1-1}{k_1}\right)^\frac{1}{d}
\frac{m_i^\frac{1}{d}}{(1-6\e_i)l_i} < \left(\frac{k_1-1}{k_1}\right)^\frac{1}{d}
\frac{\beta}{1-6\e_i}
\end{equation*}
The last quantity will be less than $\beta$ if
\begin{equation}\label{eq:i}
\left(\frac{k_1-1}{k_1}\right)^\frac{1}{d} \leq 1-6\e_i=1-\frac{6}{k \beta} \frac{m_i^\frac{1}{d}}{l_i} .
\end{equation}
In turn inequality \eqref{eq:i} is true because (iii) holds for $Q_i$, provided we choose $k\geq 6 / \big(1-(1-1/k_1)^\frac{1}{d}\big)$. 
Furthermore, estimating $\e_0<1/k$ by (iii) and (v) we  also  have $\e_{i+1}\leq\e_i$.

We are left to prove (vi): $\lim_i m_i=0$ follows directly from (iv); regarding the non degeneracy of the cubes, note that 
\begin{equation*}
\begin{split}
\frac{l_\infty}{l_0}:=\liminf_i \frac{l_i}{l_0} &\geq  \prod_{i=0}^{\infty}(1-6\e_i)=\prod_{i=0}^{\infty}\left (1-\frac{6}{k \beta}\frac{m_i^\frac{1}{d}}{l_i}\right ) \\
&\geq
\prod_{i=0}^{\infty}\left (1-\frac{6m_0^\frac{1}{d}}{k \beta l_0 \prod_{h=0}^{i-1}(1-6\e_h)}\left(\frac{k_1-1}{k_1}\right)^\frac{i}{d} \right)\\
 &\geq\prod_{i=0}^{\infty}\left (1-\frac{6}{k(1-6\e_0)^i}\left(\frac{k_1-1}{k_1}\right)^\frac{i}{d}\right),
\end{split}
\end{equation*}
where we used  $\e_h\leq \e_0$ in the last inequality. Since $\e_0<1/k$ the last product is strictly positive provided 
$$
k>\frac{6}{1-\left(\frac{k_1-1}{k_1}\right)^\frac{1}{d}} 
$$
which is guaranteed by our choice of $k$. 
We conclude that $l_\infty>0$ which ensures claim \eqref{absurd lower bound}.

\medskip

\noindent {\it Step two}: We fix $x\in \spt\, \mu \setminus H$, and prove that
\begin{equation}\label{monotonicity mu}
 r\mapsto \frac{\mu(B_{x,r})}{r^d}\quad\mbox{is increasing on $(0,d_x)$.}
\end{equation}
The proof is a straightforward adaptation of the corresponding one in \cite[Theorem 2]{DelGhiMag}, and amounts 
to prove a differential inequality for the function $f(r):=\mu(B_{x,r})$. In turn, this inequality 
is obtained in a two step approximation: first one exploits the rectifiability of the minimizing sequence $(K_j)$
 and property \eqref{inf good class} to compare $K_j$ with the cone competitor $\C_{x,r}(K_j)$, see \eqref{cone comp}. 
 The comparison a priori is only allowed with elements of $\mathcal{P}(H)$, so for almost every $r<d_x$ 
 it holds:
\begin{equation*}
\begin{split}
 f_j(r) &= \H^d(K_j)-\H^d(K_j\setminus \overline{B_{x,r}})\le m_0 + o_j(1)-\H^d(K_j\setminus \overline{B_{x,r}}) \\
 &\le o_j(1)+ \inf_{\substack{K'\in \mathcal{P}(H)}} \H^d(K') -\H^d(K_j\setminus \overline{B_{x,r}}) 
 \leq  o_j(1)+ \inf_{\substack{K'\in \mathcal{P}(H) \\ K'\setminus \overline{B_{x,r}} = K_j\setminus \overline{B_{x,r}}}} 
 \H^d(K'\cap \overline{B_{x,r}}),
 \end{split}
 \end{equation*}
 where $f_j(r):=\H^d(K_j\cap B_{x,r})$ and \(\eta_j\) is infinitesimal. Nevertheless $K_j$ can be 
 compared with its cone competitor, up to an error infinitesimal in $j$, thanks to Lemma \ref{l:10}.
 We recover
 \begin{equation*}
\begin{split}
  \inf_{\substack{K'\in \mathcal{P}(H) \\ K'\setminus \overline{B_{x,r}} = K_j\setminus \overline{B_{x,r}}}} 
  \H^d(K'\cap \overline{B_{x,r}})   &\leq o_j(1)+ \H^d(\C_{x,r}(K_j)\cap \overline{B_{x,r}})  \\ 
  &\leq o_j(1) +\frac rd \H^{d-1}(K_j\cap\pa B_{x,r})  =  o_j(1)+  \frac rd f_j'(r)\,.
 \end{split}
 \end{equation*} 
 One then passes to the limit in $j$ and obtains the desired monotonicity formula.
We refer to \cite[Theorem 2]{DelGhiMag} for the conclusion of the proof of \eqref{monotonicity mu}. 

\medskip

\noindent {\it Step three}:  By \eqref{lower density estimate mu ball} and \eqref{monotonicity mu} the $d$-dimensional density 
of the measure  $\mu$, namely:
\[
\theta(x)=\lim_{r\to 0^+}\frac{f(r)}{\omega_d r^d}\ge{\theta_0}\, ,
\]
exists, is finite and positive $\mu$-almost everywhere. Preiss' Theorem \ref{Preiss} 
implies that $\mu = \theta \H^d \res \tilde{K}$ for some countably $\H^d$-rectifiable set $K$ and some positive Borel function $\theta$.
Since $K$ is the support of $\mu$, $\H^d (\tilde{K}\setminus K) =0$. 
On the other hand, by differentiation of Hausdorff measures,  \eqref{lower density estimate mu ball} yields $\H^d (K\setminus \tilde{K}) =0$. Hence $K$ is  \(d\)-rectifiable and $\mu = \theta \H^d\res K$. 

\medskip

\noindent {\it Step four}: We prove that $\theta(x)\ge 1$ for every $x\in K$ such that the approximate tangent space
to $K$ exists (thus, $\H^d$-a.e. on $K$). Fix any such $x\in K\setminus H$ without loss of generality we can suppose that $\pi = \{x_{d+1}= ... = x_n =0\}$ 
is the approximate tangent space to $K$ at $x$, and that $x=0$: in particular,
\[
 \H^d\res \frac{K}{r} \rightharpoonup^* \H^d\res \pi\,,\qquad\mbox{as $r\to 0^+$}\,.
\]
The above convergence, together with the lower density estimates \eqref{lower density estimate mu ball} imply that, for every $\e>0$ there is $\rho>0$ such that
\begin{equation}\label{e:uniforme}
 K\cap B_{r}  \subset \left\{|y_{d+1}|,...,|y_{n}|< \frac \e2 \,r\right\} \qquad \forall r<\rho\, .
\end{equation}

Let us now assume, by contradiction,  that $\theta(0)<1$.  Thanks to \eqref{monotonicity mu} and \eqref{e:uniforme}, there exist $r \in (0,d_x)$ and $\alpha <1$ such that
\begin{equation}\label{eq:7}
\frac{\mu (Q_{\rho})}{\rho^d} \le\alpha <1, \quad \mu(\pa Q_{r}) = 0, \quad K\cap (Q_{\rho} \setminus R_{\rho,\e \rho}) = \emptyset \qquad \forall \,\rho\le r\,. 
\end{equation}
In  particular, since \(\mu_j\) are weakly converging to \(\mu\) we get that for \(j\) large
\begin{equation}\label{eq:10}
\frac{\mu_j (Q_{r})}{r^d} \le \alpha < 1 \quad \mbox{and} \quad  \mu_j(Q_{r} \setminus R_{r,\e r}) =o_j(1), 
\end{equation}
We now wish to clear the small amount of mass appearing in the complement of $R_{r,\e r}$: we achieve this by 
repeatedly applying Theorem \ref{deformationcube}. 
We set $Q_{r}\cap \{x_{d+1}\geq \frac \e2 \,r\}=:R$, and we apply Theorem \ref{deformationcube} to this rectangle with parameter $\e r$ and 
$E=K^0_j:=K_j$, obtaining the map  $\vphi_{1,j}$. We recall that the obtained constant $k_1$ for the area bound is universal, since it depends on the side ratio of $R$, which is bounded from below by $1$ and from above by $4$, provided $\varepsilon$ small enough. We set $K^1_j:=\vphi_{1,j}(K^0_j) $ and 
repeat the argument with $Q_{r}\cap \{x_{d+1}\leq- \frac \e2 \,r\}=:R$ and $E:=K^1_j$, obtaining the map $\vphi_{2,j}$. We again set $K^2_j:=\vphi_{2,j}(K^1_j) $ and iterate this procedure to the rectangles $Q_{r}\cap \{x_{d+2}\geq \frac \e2 \,r\},...,Q_{r}\cap \{x_{n}\leq- \frac \e2 \,r\}$. After $2(n-d)$  iteration, we set 
$$K_j^{2(n-d)}:= \vphi_{2(n-d),j} \circ ... \circ \vphi_{1,j} (K_j).
$$
We are going to use the cube $Q_{r(1-\sqrt{\e})}$ because, taking $\varepsilon$ small enough, then $\sqrt{\e}>4C\e$, where $C>1$ is the side ratio considered before. 
This allows us to claim that 

\begin{equation}\label{em}
\H^d(K_j^{2(n-d)}\cap (Q_{r(1-\sqrt{\e})}\setminus R_{r(1-\sqrt{\e}),6\e r}))=0.
\end{equation} 
Otherwise  there would exist a $d$-face of a smaller rectangle $T \subset (Q_{r} \setminus R_{r,\e r})$ such that 
\[
\H^d(K_j^{2(n-d)}\cap T)=\H^d(T)\geq  \e^d r^d\,,
\]
 which would lead to the following contradiction for $j$ large:
\begin{equation*}
\e^d r^d \leq \H^d(T) \leq \H^d\left( K_j^{2(n-d)} \cap (Q_{r} \setminus R_{r,\e r}) \right) \leq k_1^{2(n-d)}\H^d\left( K_j \cap (Q_{r} \setminus R_{r,\e r})  \right) =o_j(1)
\end{equation*}
In particular, we cleared any measure on every slab 
\[
\bigcup_{i=d+1}^{n}\left \{3\e r<|x_i|<(1-\sqrt{\e})\frac r2 \right \}\cap Q_{r(1-\sqrt{\e})}.
\]
We want now to  construct a map $P \in  \D(0,r)$, collapsing $R_{r(1-\sqrt{\e}),6\e r}$  onto the tangent plane.  To this end, for  $x\in \R^n$,  $x=(x',x'')$ with $x'\in \R^d$ and $x''\in \R^{n-d}$ we set
\begin{equation}\label{infinito}
\|x'\|:=\max\{|x_i|: i=1,\dots,d\}\qquad \|x''\|:=\max\{|x_i|:  i=d+1,\dots,n\}
\end{equation}
and we define  $P$ as follows:
\begin{equation} \label{P}
P(x)=
\begin{cases}
\big(x', g(\|x'\|)\frac{(\|x''\|-3\e r)_+}{1-6\e}\frac{x''}{\|x''\|}+(1-g(\|x'\|))x''\big)\quad&\textrm{if \(\max\{\|x'\|,\|x''\|\}\le r/2\)}\\
\Id&\textrm{otherwise,}
\end{cases}
\end{equation}
where \(g: [0,r/2]\to [0,1]\) is a compactly supported cut off function such that
\[
g\equiv 1\quad \textrm{on \([0,r(1-\sqrt{\e})/2]\)}\qquad \textrm{and} \qquad |g'|\le 10/r\sqrt{\e}\,.
\]
 It is not difficult to check that \(P\in \D(0,r)\)  and that \({\rm Lip}  \,P\le 1+C\sqrt{\e}\) for some dimensional constant \(C\).

We now set $\widetilde{K_j}:=P(K_j^{2(n-d)})$, which verifies, thanks to \eqref{em},
\begin{equation}\label{empty}
\H^d\Big(\widetilde{K_j}\cap \big(Q_{(1-\sqrt{\e})r}\setminus Q_{(1-\sqrt{\e})r}^d\big)\Big)=0
\end{equation}
and
\begin{equation}\label{small}
\begin{split}
&\H^d\big(\widetilde{K_j}\cap \big(Q_{r} \setminus Q_{r(1-\sqrt{\e})}\big)\big) \leq (1+C\sqrt{\e})^d \H^d\big(K_j^{2(n-d)}\cap \big(Q_{r} \setminus Q_{r(1-\sqrt{\e})}\big)\big) \\
&\leq (1+C\sqrt{\e})k_1^{2(n-d)} \H^d\big(K_j\cap \big(Q_{r} \setminus (Q_{r(1-\sqrt{\e})}\cup R_{r,\e r})\big)\big) \\
& \quad +(1+C\sqrt{\e})\H^d\big(K_j\cap \big(R_{r,\e r} \setminus Q_{r(1-\sqrt{\e})}\big)\big) \\
&\leq o_j(1)+ (1+C\sqrt{\e})\H^d\big(K_j\cap \big(R_{r,\e r} \setminus Q_{r(1-\sqrt{\e})}\big)\big)\,,
\end{split}
\end{equation}
where in the last inequality we have used \eqref{eq:10}. Moreover, by using \eqref{eq:7},  \eqref{eq:10} and \eqref{empty} we also have that, for \(\e\) small and  \(j\) large:
\begin{equation}\label{dens}
\begin{split}
\frac{\H^d(\widetilde{K_j}\cap Q^d_{r(1-\sqrt{\e})})}{r^d(1-\sqrt{\e})^d}=\frac{\H^d(\widetilde{K_j}\cap Q_{r(1-\sqrt{\e})})}{r^d(1-\sqrt{\e})^d} &\leq (1+C\sqrt{\e})\frac{\H^d(K_j^{2(n-d)}\cap Q_{r})}{r^d}\\
&\leq (1+C\sqrt{\e}) \frac{\H^d(K_j\cap Q_{r})+o_j(1)}{r^d} \\
&\le \alpha + o_j(1) < 1.
\end{split}
\end{equation}
As a consequence of \eqref{dens} and the compactness of $\widetilde{K_j}$, there exist $y'_j \in Q_{(1-\sqrt{\e})r}^d$ and $\delta_j >0$ such that, if we set $y_j:=(y'_j, 0)$, then
\begin{equation}\label{eq:11}
\widetilde{K_j} \cap B^{d}_{y_j,\delta_j} = \emptyset \quad \mbox{and} \quad B^{d}_{y_j,\delta_j}\subset Q_{(1-\sqrt{\e})r}^d.
\end{equation}
After the last deformation, our set $\widetilde{K_j}\cap Q_{r(1-\sqrt{\e})}$ lives on the tangent plane and we want to use the property \eqref{eq:11} to collapse $\widetilde{K_j}\cap Q_{r(1-\sqrt{\e})}$ into $\left (\pa Q_{(1-\sqrt{\e})r}^d\right ) \times \{0\}^{n-d}$. To this end let us  define for every $j \in \N$, the following Lipschitz map:
\begin{equation*}
\varphi_j(x)= 
\begin{cases}
 \big (x'+ z'_{j,x},\, x''\big) & \mbox{if } x \in R_{r(1-\sqrt{\e}), r}\\ 
 x & \mbox{otherwise},
\end{cases}
\end{equation*}
with 
$$ z'_{j,x} := \min \left \{1,\frac{\left |x'-y'_j\right |}{\delta_j}\right \}\frac{\left (r-4\|x''\|\right)_+}{r}\gamma_{j,x}(x'-y'_j),$$
where $\gamma_{j,x}>0$ is such that $x'+\gamma_{j,x}(x'-y'_j) \in \pa Q_{(1-\sqrt{\e})r}^d \times \{0\}^{n-d}$ and \(\|x''\|\) is defined in \eqref{infinito}. One can easily check that $\varphi_j \in  \D(0,r)$. Moreover, setting  $\varphi_j(\widetilde{K_j})=: K'_j$ we have that 
$$K'_j \setminus Q_{r} = K_j \setminus Q_{r}$$
and 
\begin{equation}\label{:)}
\H^d(K_j'\cap Q_{r(1-\sqrt{\e})})=0,
\end{equation}
thanks to  \eqref{empty}, since 
\[
\H^d\Big (\pa Q_{(1-\sqrt{\e})r}^d \times \{0\}^{n-d}\Big )=0.
\] 
Since \(\mathcal{P} (H)\) is a good class, by \eqref{inf good class} there exists a sequence of competitors $(J_j)_{j \in \N} \subset \mathcal{P} (H)$ such that $J_j\setminus \overline B_{0,r} =K_j\setminus \overline
B_{0,r} $ and $\H^d (J_j)= \H^d (K'_j)+o_j(1)$. Hence, thanks to \eqref{small} and \eqref{:)} we get
\begin{equation*}
\begin{split}
\H^d(K_j)-\H^d(J_j) &\geq \H^d(K_j)-\H^d(K'_j) -o_j(1) = \H^d(K_j\cap Q_{r})-\H^d(K'_j\cap Q_{r}) -o_j(1)\\
&\geq \H^d\left(K_j\cap Q_{r(1-\sqrt{\e})} \right) + \H^d\left(K_j\cap (R_{r,\e r} \setminus Q_{r(1-\sqrt{\e})}) \right) + \\
&\quad-o_j(1)-(1+C\sqrt{\e})\H^d\left(K_j\cap (R_{r,\e r} \setminus Q_{r(1-\sqrt{\e})}) \right)\\
&\geq  \H^d\left(K_j\cap Q_{r(1-\sqrt{\e})} \right)-C\sqrt{\e}\H^d\left(K_j\cap (R_{r,\e r} \setminus Q_{r(1-\sqrt{\e})}) \right)-o_j(1).
\end{split}
\end{equation*}
Passing to the limit as \(j\to \infty\), and using \eqref{muj va a mu}, \eqref{lower density estimate mu ball} and \eqref{eq:7} we get that
\[
\begin{split}
\liminf_{j} \H^d(K_j)&\ge \liminf_{j} \H^d(J_j)+\mu(Q_{r(1-\sqrt{\e})})-C\sqrt{\e}r^d\\
&\ge \liminf_{j} \H^d(J_j)+(\theta_0(1-\sqrt{\e})^d-C\sqrt{\e})r^d.
\end{split}
\]
Since, for \(\e\)  small,  this is in contradiction with \(K_j\) be a minimizing sequence we finally conclude that \(\theta(0)\ge 1\).
 
 \medskip

\noindent {\it Step five}: We prove that $\theta(x)\leq 1$ for every $x\in K$ such that the approximate tangent space
to $K$ exists. 
Arguing by contradiction, we assume that $\theta(x)=1+\sigma>1$ for some $x$ where $K$ admits an approximate tangent plane $T$. As usually we assume that $x=0$ and $T = \{y: y_{d+1},...,y_{n} =0\}$. By the monotonicity of the density  established in  Step 2, for every \(\e>0\)
we can find $r_0>0$ such that
  \begin{equation}
    \label{mancoilnome}
      K\cap Q_{r}\subset R_{r,\e r}\,,\qquad 1+\sigma\le\frac{\mu(Q_{r})}{r^d}\le1+\s+ \e\,\s\,, \qquad \forall\, r\le r_0\, .
  \end{equation}
Let us also  note that for  every $r<r_0$ there exists  $j_0 $ such that
  \begin{equation}
    \label{buonpunto}
      \H^d(K_j\cap Q_{r})>\Big(1+\frac{\s}2\Big)\,r^d\,,\qquad \H^d( (K_j\cap Q_{r})\setminus R_{r,\e r})<\frac\s{4}\,r^d,\qquad\forall j\ge j_0\,,
  \end{equation}
 Consider the map $P:\R^{n}\to\R^{n}\in \D(0,r)$ with $\Lip P\le 1+C\,\sqrt\e$ defined in \eqref{P} 
 which collapses \(R_{(r(1-\sqrt{\e}),\e r}\) onto the tangent plane. By exploiting the fact the \(\mathcal P(H)\) is a good class we find that
  \begin{equation*}
  \begin{split}
    \H^d(K_j\cap Q_{r})-o_j(1) &\le
    \underbrace{\H^d( P(K_j\cap R_{(1-\sqrt{\e})r,\e r}))}_{I_1}
    +
    \underbrace{\H^d( P(K_j\cap(R_{r,\e r}\setminus R_{(1-\sqrt{\e})r,\e r})))}_{I_2} \\
    &+
    \underbrace{\H^d( P(K_j\cap (Q_{r}\setminus R_{r,\e r})))}_{I_3}\,.
    \end{split}
  \end{equation*}
  By construction, $I_1\le r^d$, while, by \eqref{buonpunto}, $\H^d(K_j\cap Q_{r})>(1+(\s/2))r^d$ and
  \[
  I_3\le(\Lip P)^d\,\H^d(K_j\cap (Q_{r}\setminus R_{r,\e r}))<(1+C\,\sqrt\e)^d\,\frac{\s}{4}\,r^d\,.
  \]
  Hence, as $j\to\infty$,
  \[
  \Big(1+\frac{\s}2\Big)r^d\le r^d+\liminf_{j\to\infty}I_2+(1+C\,\sqrt\e)^d\,\frac{\s}{4}\,r^d\,,
  \]
  that is,
  \begin{equation}\label{this}
  \Big(\frac12-\frac{(1+C\,\sqrt\e)^d}4\Big)\,\s\le\liminf_{j\to\infty}\frac{I_2}{r^d}\,.
  \end{equation}
 By \eqref{mancoilnome}, we finally estimate that
  \begin{eqnarray}\nonumber
    \limsup_{j\to\infty}\,I_2&\le&(1+C\sqrt{\e})^d\,\mu(Q_{r}\setminus Q_{(1-\sqrt\e)r})
    \\\label{that}
    &\le&(1+C\sqrt{\e})^d\Big((1+\s+\e\s)-(1+\s)(1-\sqrt\e)^d\Big)\,r^d
  \end{eqnarray}
  By choosing \(\e\) sufficiently small, \eqref{this} and \eqref{that} provide the desired contradiction. Hence \(\theta \le 1\) and, combining this with the previous step, \(\mu=\H^{d}\res K\).

\medskip

\noindent {\it Step six}: We now show that the canonical density one rectifiable varifold associated $K$ is stationary in $\mathbb R^{n}\setminus H$. In particular, applying Allard's regularity theorem, see  \cite[Chapter 5]{SimonLN}, we will deduce that there exists an $\H^d$-negligible closed set $\Sigma\subset K$ such that $\Gamma = K \setminus \Sigma$ is a real analytic  manifold.  Since   being a stationary varifold is a local property, to prove our claim it is enough to show that for every ball \(B \cc \R^n\setminus H\) we have
 \begin{equation}\label{diffeomin}
 \H^d(K)\le\H^d(\phi(K))
 \end{equation}
  whenever $\phi$ is a diffeomorphism such that \(\spt\{\phi-\Id\}\subset B\). Indeed,  by exploiting \eqref{diffeomin} with  \(\phi_t=\Id+tX\), \(X\in C_c^1(B)\) we deduce the desired stationarity property.
  
 To prove \eqref{diffeomin} we argue as in \cite[Theorem 7]{DelGhiMag}.  Given $\e>0$ we can find $\de>0$ and a compact set $\hat{K} \subset K\cap B$ with $\H^d(K\setminus \hat{K})\cap B )<\e$ such that $K$ admits an approximate tangent plane $\pi(x)$ at every $x\in\hat K$,
 \begin{eqnarray}
   \label{approx cont}
    \sup_{x\in \hat{K}}\sup_{y\in B_{x,\de}}|\nabla\phi(x)-\nabla\phi(y)|\le \e\,,\qquad
 \sup_{x\in \hat{K}}\sup_{y\in \hat{K}\cap B_{x,\de}}{\rm d}(\pi(x),\pi(y))<\e\,,
 \end{eqnarray}
 where \({\rm d}\) is a distance on \(G(d)\), the \(d\)-dimensional Grassmanian. Moreover, denoting by $S_{x,r}$ the set of points in $B_{x,r}$ at distance at most $\e\,r$ from $x+\pi(x)$, then $K\cap B_{x,r}\subset S_{x,r}$ for every $r<\de$ and $x\in\hat{K}$. By Besicovitch covering theorem we can find a finite disjoint family of closed balls $\{\overline B_i\}$ with $B_i=B_{x_i,r_i}\subset B\cc \R^{n}\setminus H$, $x_i\in \hat{K} $, and $r_i<\de$, such that $\H^d(\hat{K} \setminus\bigcup_iB_i)<\e$. By exploiting the construction of Step four, we can find $j(\e)\in\N$ and maps $P_i: \R^n\to \R^n$ with $\Lip(P_i)\le 1+C\,\sqrt{\e}$ and $P_i=\Id$ on \(B_i^c\), such that, for a certain $X_i\subset S_i=S_{x_i,\e\,r_i}$,
 \begin{equation}   \label{ober1}
 \begin{aligned}
   &P_i(X_i)\subset B_i\cap(x_i+\tau(x_i))\,,
   \\
   &\H^d\Big(P_i\big((K_j\cap B_i)\setminus X_i\big)\Big)\le C\,\sqrt\e\,\om_d\,r_i^d\,,\qquad\forall j\ge j(\e)\,.
 \end{aligned}
 \end{equation}
Exploiting  \eqref{approx cont}, \eqref{ober1},  the area formula and that   $\om_d\,r_i^d\le\H^d(K\cap B_i)$ (by the monotonicity formula), and setting $\a_i=\H^d((K\setminus \hat K)\cap B_i)$, we get (denoting with \(J_d^{\pi}\) the \(d\)-dimensional tangential jacobian with respect to the  plane \(\pi\) and by \(J^K_d\) the one with respect to \(K\))
 \begin{equation}\label{diavolo}
 \begin{split}
    \H^d(\phi(P_i(K_j\cap X_i)))&=\int_{P_i(K_j\cap X_i)}J_d^{\pi(x_i)}\phi(x)\,d\H^d(x)\le(J^{\tau(x_i)}\phi(x_i)+\e)\,\om_d\,r_i^d
   \\
   &\le(J_d^{\pi(x_i)}\phi(x_i)+\e)\,\H^d(K\cap B_i)
   \le(J_d^{\pi(x_i)}\phi(x_i)+\e)\,(\H^d(\hat K\cap B_i)+\a_i)
   \\ 
   &\le\int_{\hat K\cap B_i}(J_d^{K}\phi(x)+2\e)\,\,d\H^d(x)+((\Lip\phi)^d+\e)\,\a_i
   \\
   &=\H^d(\phi(\hat K\cap B_i))+2\e\,\H^d(K\cap B_i)+((\Lip\phi)^d+\e)\,\a_i\,,
 \end{split}
 \end{equation}
 where in the last identity we have used the area formula and the injectivity of $\phi$.  Since \(P_i=\Id\) on \(B_i^c\), \(\phi=\Id\) on \(B^c\),  \(B_i\subset B\) and the balls \(B_i\) are disjoint, the map \(\tilde \phi\) which is equal to \(\phi\) on \(B\setminus \cup_{i} B_i\), equal to the identity on \(B^c\) and equal to \(\phi\circ P_i\) on \(B_i\) is well defined. Moreover, by \eqref{diavolo}, we get
 \[
 \H^d(\tilde \phi(K_j))\le \H^d(\phi(K))+C\e
 \] 
 where \(C\) depends only on \(K\). By exploiting the definition of good class, we get that 
 \[
 \H^d(K)\le \H^d(\tilde \phi(K_j))+\eta_j\le \H^d(\phi(K))+C\e+o_j(1).
 \]
 Letting \(j\to \infty\) and \(\e\to 0\) we obtain \eqref{diffeomin}.

\medskip
\noindent {\it Step seven}: We finally address the dimension of the singular set. Recall that, by monotonicity, the density function 
\[
\Theta^d(K,x)=\lim_{r\to 0}\frac{\H^d(K\cap B_{x,r})}{\omega_d r^d}
\]
is everywhere defined in $\R^n\setminus H$ and equals $1$ $\H^d$-almost everywhere in $K$. 
 Fixing $x\in K$ and a sequence $r_k\downarrow 0$, the monotonicity formula, 
 the stationarity of $\H^d\res K$ and the compactness theorem for integral varifolds 
 \cite[Theorem 6.4]{Allard} imply that (up to subsequences)
 \begin{equation}\label{tanvar}
 \H^d\res\left(\frac{K-x}{r_k}\right)\rightharpoonup V\quad\mbox{ locally in the sense of varifolds,}
 \end{equation}
with
\begin{itemize}
\item[(a)] $V$ is a stationary integral varifold: in particular $\Theta^d(\|V\|,y)\geq 1$ for $y\in\spt(V)$;
\item[(b)] $V$ is a cone, namely $(\delta_{\lambda})_{\#}V=V$, where  $\delta_{\lambda}(x) = \lambda x$, $\lambda>0$;
\item[(c)] $\Theta^d(\|V\|,0)  = \Theta^d(K,x) \geq \Theta^d(\|V\|,y)$ for every $y\in\R^n$.
\end{itemize}
Recall that the tangent varifold $V$ depends (in principle) on the sequence $(r_k)$.
We denote by ${\rm TanVar}(K,x)$ the 
(nonempty) set of all possible limits $V$ as in \eqref{tanvar} varying among all sequences along which \eqref{tanvar} holds. 
Given a cone $W$ we set
\begin{equation}\label{spine}
{\rm Spine}(W):=\{y\in\R^n : \Theta^d(\|W\|,y) = \Theta^d(\|W\|,0) \}:
\end{equation}
by \cite[2.26]{almgrenBIG} ${\rm Spine}(W)$ is a vector subspace of $\R^n$, see also \cite{whitestrat}. 
We can stratify $K$ in the following way: for every $k=0,\dots,n$ we let 
\begin{equation*}
A_k := \{x \in K : \textrm{ for all } V\in {\rm TanVar}(K,x),\,\dim {\rm Spine}(V)\leq k\}.
\end{equation*}
Clearly  $A_0\supset\dots\supset A_{d+1}=\dots=A_n=\emptyset$; 
moreover it holds: $\dim_\H A_k\leq k$, see \cite[2.28]{almgrenBIG}. 
In order to prove our claim, we need to show that $A_d\setminus A_{d-1}\subset K\setminus \Sigma$, 
where $\Sigma$ as in Step six is the singular set of $K$, 
namely that every point in $K$ having at least one tangent cone of maximal dimension $d$ must be regular (note that, as the example of complex varieties shows, this is not true in general for a stationary varifold). 

First, note that if $x\in A_d\setminus A_{d-1}$ and $V\in{\rm TanVar}(K,x)$ satisfies $\dim {\rm Spine}(V)=d$, then $V = \H^d\res {\rm Spine}(V)$: indeed up to a 
rotation $\spt(V) = {\rm Spine}(V)\times \Gamma$, where $\Gamma$ is a cone in $\R^{n-d}$. If $\Gamma\neq\{0\}$ 
then $\Theta^d(\|V\|,0)>\Theta^d(\|V\|,y)$ for any $y\in {\rm Spine}(V)\setminus\{0\}$, which contradicts \eqref{spine}. Hence
by \eqref{spine} and (c), $\Theta^d(\|V\|,0)$ is an integer 
and $V = \Theta^d(K,x)\H^d\res {\rm Spine}(V)$ (see also \cite[Theorem 2.26 (4)]{almgrenBIG}). 
Second, the density lower bound $\Theta^d(K,\cdot)\geq 1$ and \eqref{tanvar} imply that for every  $\e>0$ 

$$
\frac{K-x}{r}\subset U_\e({\rm Spine}(V))\quad\mbox{and}\quad
\H^d\left(\frac{K-x}{r}\cap Q_{0,1}\right)\leq (1+\e)\Theta^d(K,x)
$$
if $r$ is sufficiently small. By arguing as in Step five (i.e. roughly speaking comparing $K$ with $P(K)$ in $Q_{x,r}$, where $P$ is the squeezing map \eqref{P} although one has to rigorously get through the minimizing sequence \(K_j\)) 
we obtain 
$$
\H^d\left(\frac{K-x}{r}\cap Q_{0,1}\right)\leq (1+C\sqrt{\e}).
$$
Letting $r\downarrow 0$ thanks to \eqref{tanvar} we obtain 
$\|V\|(Q_{0,1})\le (1+C\sqrt{\e})$, implying $\Theta^d(K,x) =1$. 
We therefore fall into the hypotheses of Allard's regularity Theorem 
\cite[Regularity Theorem, Section 8]{Allard}, $K\cap Q_{x,\frac{r}{2}}$ is a real analytic 
submanifold. Equivalently $x\not\in \Sigma$.
\end{proof}

\section{Proof of Theorems \ref{thm plateau} and \ref{thm david}}\label{harrisonanddavid}

In this Section we prove Theorem \ref{thm plateau} and \ref{thm david}. With Theorem \ref{thm generale} at hand, the proofs are quite similar  to the corresponding ones in \cite{DelGhiMag} (see Theorems 4 and 7 there), hence  we limit ourselves to provide a short sketch  underlying only the main differences. 

\begin{proof}
[Proof of Theorem \ref{thm plateau}] 
We start by proving that $\F(H,\CC)$ is a good class in the sense of Definition \ref{def good class}:  
let $\widetilde{K}\in\F (H,\CC)$, $x\in K$, $r\in (0, \dist (x, H))$ and $\varphi \in \D(x,r)$. 
We show that $\vphi(\widetilde{K})\in\F(H,\CC)$ arguing by contradiction: assume that $\g(S^{n-d})\cap \vphi(\widetilde{K}) = \emptyset$
for some $\g\in\CC$ and, without loss of generality, suppose also that $\g(S^{n-d})\cap(\widetilde{K}\setminus B_{x,r})=\emptyset$. 
By Definition \ref{d:deform} there exists a sequence
$$
(\vphi_{j})\subset \mathfrak D(x,r)\qquad\mbox{ such that }\qquad\lim_{j}\|\vphi_{j} - \vphi\|_{C^0}=0.
$$ 
Since $\g(S^{n-d})$ is compact and $\vphi_j=\mbox{Id}$ outside $ B_{x,r}$, for $j$ sufficiently large $\g(S^{n-d})\cap \vphi_j(\widetilde{K}) = \emptyset$; moreover $\vphi_j$ is invertible, hence  $\vphi_{j}^{-1}\left(\g(S^{n-d})\right )\cap \widetilde{K} = \emptyset$. 
But the property for $\vphi_{j}$ of being isotopic to the identity implies $\vphi_{j}^{-1} \circ \g \in \CC$, which contradicts $V\in\F (H,\CC)$. This proves (a).

Given a minimizing sequence $(K_j )\subset\F(H,\CC)$ which consists of rectifiable sets, we can therefore find 
a set $K$ with the properties stated in Theorem \ref{thm generale}. In order to conclude (b), namely that $K\in\F(H,\CC)$, 
we refer to \cite[Theorem 4(b)]{DelGhiMag}: the proof is the same.

\end{proof}

\begin{proof}[Proof of Theorem \ref{thm david}]  As already observed in Remark \ref{remark david}, $\A (H, K_0)$ is a good class  and we can therefore apply Theorem \ref{thm generale}. We thus know that $\H^d\res K_j\weak\mu= \ \H^d\res K\) and that  \(K\) is a smooth set away from \(H\) and from a relatively closed set \(\Sigma\)  of dimension less or equal than \((d-1)\). The conclusion of the proof can now be obtained by repeating verbatim Steps 4 and 6 in the proof of Theorem 7 in \cite{DelGhiMag}.

\end{proof}

\end{document}